\documentclass[11pt]{amsart}

\usepackage{amsfonts, amssymb, amscd}
\numberwithin{equation}{section}

\usepackage[symbol]{footmisc}
\renewcommand{\thefootnote}{\fnsymbol{footnote}}

\usepackage{bm}
\usepackage{verbatim}
\usepackage{amssymb}
\usepackage{mathrsfs}
\usepackage{graphicx}
\usepackage{tikz-cd}
\usepackage{subcaption}
\usepackage{listings}
\usepackage{subfiles}
\usepackage[toc,page]{appendix}
\usepackage{mathtools}
\usepackage{comment}
\usepackage{enumerate}
\usepackage{enumitem}
\usepackage[linesnumbered,ruled]{algorithm2e}

\usepackage{graphicx}
\graphicspath{{images/}}

\usepackage{appendix}
\usepackage{hyperref}
\newcommand{\Qq}{\mathbb{Q}}

\newcommand{\Rr}{\mathbb{R}}

\newcommand{\Zz}{\mathbb{Z}}

\newcommand{\Center}{\operatorname{center}}

\newcommand{\mld}{{\rm{mld}}}

\newcommand{\Ii}{{\Gamma}}

\newcommand\Mod{{\rm{mod}}}

\newtheorem{thm}{Theorem}[section]
\newtheorem{conj}[thm]{Conjecture}

\newtheorem{lem}[thm]{Lemma}
\newtheorem{defn}[thm]{Definition}

\newtheorem{claim}[thm]{Claim}
\theoremstyle{definition}
\newtheorem{rem}[thm]{Remark}

\newtheorem{ex}[thm]{Example}

\theoremstyle{definition}

\begin{document}

\title{An optimal gap of minimal log discrepancies of threefold non-canonical singularities}

\author{Jihao Liu and Liudan Xiao}

\address{Department of Mathematics, The University of Uath, Salt Lake City, UT 84112, USA}
\email{jliu@math.utah.edu}

\address{Google LLC, 1600 Amphitheater Parkway, Mountain View, California, U.S.A.}
\email{liudan.xiao.pku@gmail.com}

\begin{abstract}
	We show that the minimal log discrepancy of any $\Qq$-Gorenstein non-canonical threefold is $\leq\frac{12}{13}$, which is an optimal bound.
\end{abstract}
\let\thefootnote\relax

\date{\today}

\maketitle
\pagestyle{myheadings}\markboth{\hfill  J.Liu and L.Xiao \hfill}{\hfill An optimal gap of minimal log discrepancies of threefold non-canonical singularities\hfill}

\tableofcontents

\section{Introduction}

We work over the filed of complex numbers $\mathbb C$.

The minimal log discrepancy (mld for short), which was introduced by Shokruov, is an important algebraic invariant of singularities and plays a fundamental role in birational geometry. For simplicity, we only use a very simple version of mlds in this paper.

\begin{defn}
Let $X$ be a $\mathbb Q$-Gorenstein normal variety and $x\in X$ a closed point. The minimal log discrepancy of $X$ at $x$ is defined as
$$\mld(x,X):=\min\left\{a(E,X)|E\text{ is a prime divisor over } X,\Center_XE=\left\{x\right\}\right\},$$
and the \emph{minimal log discrepancy} of $X$ is defined as
$$\mld(X):=\min\left\{a(E,X)|E\text{ is a prime divisor over } X\right\}.$$
\end{defn}
It is conjectured by Shokurov that the set of minimal log discrepancies satisfies the ascending chain condition (ACC).
\begin{conj}\label{conj: acc mld no boundary}
Let $d$ be a positive integer. Then $$\left\{\mld(X)|\dim X=d\right\}$$ satisfies the ACC.
\end{conj}

Conjecture \ref{conj: acc mld no boundary} is known for curves and surfaces by Alexeev \cite{Ale93} and Shokurov \cite{Sho94}, and for toric varieties by Borisov \cite{Bor97} and Ambro \cite{Amb06}. For other related results, we refer the readers to \cite{Sho96,Sho00,Sho04,Kaw11,Kaw14,Kaw15,MN18,Nak16,Kaw18,Liu18,HLS19,CH20,HL20}. 

In particular, by classification of terminal threefold singularities, Conjecture \ref{conj: acc mld no boundary} is well-known for canonical threefolds. However, for non-canonical singularities, Conjecture \ref{conj: acc mld no boundary} had been widely open in dimension $\geq 3$ for decades. In \cite{Jia19}, Jiang shows the following result, which is the first step towards proving Conjecture \ref{conj: acc mld no boundary} for threefold non-canonical singularities:

\begin{thm}[{\cite[Theorem 1.3]{Jia19}}]\label{thm: jiang3dimmld}
There exists a positive real number $\delta$, such that for any normal $\mathbb Q$-Gorenstein variety $X$ of dimension $3$, if $\mld(X)<1$, then $\mld(X)\leq 1-\delta$.
\end{thm}

 \cite{Jia19} does not give an explicit positive lower bound of $\delta$. In this paper, we show that the optimal bound of $\delta$ is $\frac{1}{13}$.
 
\begin{thm}\label{thm: 12/13}
For any normal $\mathbb Q$-Gorenstein variety $X$ of dimension $3$, if $\mld(X)<1$, then $\mld(X)\leq\frac{12}{13}$. In particular, Conjecture \ref{conj: acc mld no boundary} holds when $d=3$ and $\mld(X)\geq\frac{12}{13}$.
\end{thm}

The following example shows that we cannot take $\delta>\frac{1}{13}$:

\begin{ex}\label{ex: 12/13}
The three-dimensional cyclic quotient singularity $(x\in X):=\frac{1}{13}(3,4,5)$ has minimal log discrepancy $\frac{12}{13}$.
\end{ex}

The proof of Theorem \ref{thm: 12/13} is based on a more precise study on $5$-dimensional toroidal singularities with the help of computer programs and by applying Jiang's proof of Theorem \ref{thm: jiang3dimmld} in \cite{Jia19}. More precisely,  \cite{Jia19} shows that any non-canonical singularity is associated with an extremely non-canonical singularity, and any three-dimensional extremely non-canonical singularity is associated with a cyclic quotient singularity of dimension $5$ with minimal log discrepancy less than $2$.  \cite{Jia19} only uses the fact that the minimal log discrepancies of $5$-dimensional cyclic quotient singularities do not have $2$ as an accumulation point from below. However, as the associated $5$-dimensional cyclic quotient singularities possess other properties, we may characterize their minimal log discrepancies more precisely. We have the following result, which is the key ingredient to prove Theorem \ref{thm: 12/13}:

\begin{thm}\label{thm: main theorem special 5dimcyc}
Let $(x\in X):=\frac{1}{r}(a_1,a_2,a_3,a_4,a_5)$ be a $5$-dimensional cyclic quotient singularity satisfying the following:
\begin{itemize}
\item $r,a_1,a_2,a_3,a_4,a_5$ are positive integers,
\item $a_i\leq r$ for any $i\in\{1,2,3,4,5\}$,
\item $\gcd(a_1,r)=\gcd(a_2,r)=\gcd(a_3,r)=1$, 
\item $\gcd(a_4,r)=\gcd(a_5,r)$,
\item $\gcd\left(\sum_{i=1}^5a_i,r\right)=1$,
\item $$2-\frac{1}{13}<\mld(x,X)=\frac{1}{r}\sum_{i=1}^5a_i<2,$$ and
\item one of the following holds: 
\begin{itemize}
\item $r\mid a_1+a_2+a_5$, or
\item $r\mid 2a_1+a_5$, or
\item $r\mid 2a_4+a_5$ and $\gcd(a_4,r)=\gcd(a_5,r)\leq 2$.
\end{itemize}
\end{itemize}
Then
\begin{align*}
    (r,\{a_1,a_2,a_3,a_4,a_5\})\in\big\{&(19,\{3,4,5,7,18\}),(17,\{2,3,5,7,16\}),\\
    &(14,\{3,5,13,2,4\})\big\}.
\end{align*}
\end{thm}

The next theorem is a by-product of the proof of Theorem \ref{thm: main theorem special 5dimcyc}, which gives an explicit bound of the minimal log discrepancies of $5$-dimensional isolated cyclic quotient singularities away from $2$. As this bound is also optimal, we write it here.
\begin{thm}\label{thm: isolated5 1/19}
Let $(x\in X)$ be an isolated $5$-dimensional cyclic quotient singularity such that $\mld(x,X)<2$. Then $\mld(x,X)\leq 2-\frac{1}{19}$.
\end{thm}

\begin{ex}
The $5$-dimensional cyclic quotient singularity $(x\in X):=\frac{1}{19}(3,4,5,7,18)$ has minimal log discrepancy $2-\frac{1}{19}$.
\end{ex}

\noindent\textit{Idea of the proof of Theorem \ref{thm: 12/13}}. Following the proof of Jiang in \cite{Jia19}, to prove Theorem \ref{thm: 12/13}, we only need to study the behavior of special cyclic quotient singularities of dimension $5$ with minimal log discrepancies $<2$ but sufficiently close to $2$. That is, we only need to prove Theorem \ref{thm: main theorem special 5dimcyc}.

Suppose that $(x\in X)=\frac{1}{r}(a_1,a_2,a_3,a_4,a_5)$ is a cyclic quotient singularity of dimension $5$ and $x_i:=\frac{a_i}{r}$ for every $i$. Our key observation is that, when $\mld(x,X)$ is sufficiently close to $2$, the function $f:\mathbb N^+\rightarrow\mathbb N$ given by
$$f(n):=\sum_{i=1}^5\lfloor nx_i\rfloor$$
has some very nice properties. In particular, we want to consider equations of the form
$$f(n)=2n-2-c.$$
If $f(n)=2n-2-c$, we say that \emph{condition $\mathcal{D}(n,c)$} holds. 

By using some tricky but elementary calculations, we know that $\mathcal{D}(n,c)$ holds for a lot of specific $n$ and $c$ at the same time. The precise behavior we need for $f(n)$ and $\mathcal{D}(n,c)$, as well as an auxiliary inequality $\mathcal{C}(n)$ which is also related to $f(n)$, are illustrated in Lemma \ref{lem: main lemma delta odd}. Finally, we use computer algorithm to enumerate the solutions $(x_1,x_2,x_3,x_4,x_5)$ for all $\mathcal{D}(n,c)$. After multiple enumerations, we eventually get a contradiction by showing that solutions for the required equations do not exist.

\medskip

\noindent\textit{Sketch of the paper}. The main of this paper involves elementary but complicated calculations on special $5$-dimensional cyclic quotient singularities. In Section 2, we give some elementary lemmas, introduce the auxiliary function $f(n)$, the auxiliary equations $\mathcal{D}(n,c)$ and the auxiliary inequalities $\mathcal{C}(n)$. In particular, we prove two technical but important lemmas (Lemma \ref{lem: main lemma using trick} and Lemma \ref{lem: main lemma delta odd}), which will be repeatedly used in the rest of the proof. In Section 3, we state the ``mathematical version" of several theorems we obtained by computer programs. In Section 4, we prove Theorem \ref{thm: main theorem special 5dimcyc}. In Section 5, we prove Theorem \ref{thm: isolated5 1/19}. In Section 6, we prove Theorem \ref{thm: 12/13}. In the appendix, we provide algorithms used to prove the theorems in Section 3 and give an explanation of a typical one of them. For detailed programs of the algorithms, we refer the readers to the arXiv version of this paper \cite{LX19}.

\medskip


\section{Prerequisites}

\subsection{Elementary definitions and lemmas}

\begin{defn}\label{defn:iq}
For any integer $q\geq 2$, we define
$\Ii_q:=\left\{n\in\mathbb N^+|\ q\nmid n\right\}.$
\end{defn}

\begin{defn}
Let $m,r$ be two positive integers, $a_1,\dots,a_m$, $a_1',\dots,a_m'$ and $e,e'$ integers, and $x_1,\dots,x_m,x_1',\dots,x_m'$ real numbers. We write
\begin{itemize}
    \item $$x_1\equiv x_1' (\Mod\ \Zz)$$ 
    if $x_1-x_2\in\Zz$,
    \item $$(x_1,\dots,x_m)\equiv (x_1',\dots,x_m') (\Mod\ \Zz)$$
    if $x_i\equiv x_i' (\Mod\ \Zz)$ for every integer $i\in [1,m]$, and
    \item $$(\left\{a_1,\dots,a_m\right\},e)\equiv (\left\{a_1',\dots,a_m'\right\},e') (\Mod\ r)$$ 
    if 
\begin{itemize}
    \item possibly reordering $a_1',\dots,a_m'$, $a_i\equiv a_i' (\Mod\ r)$ for every integer $i\in [1,m]$, and
    \item $e\equiv e' (\Mod\ r)$.
\end{itemize}
\end{itemize}
\end{defn}

The next lemma is elementary so we omit the proof.

\begin{lem}\label{lem: round down plus round up}
For any $a,b\in\Rr$ such that $a+b\in\mathbb Z$,
$a+b=\lfloor a\rfloor+\lceil b\rceil.$
\end{lem}

\subsection{Special cyclic quotient singularities}

In this subsection, we introduce special cyclic quotient singularities of dimension $5$ we need to study in this paper.

The next lemma is based on basic toric geometry. We refer the readers to \cite{Amb06} for a proof.
\begin{lem}\label{lem: mld of toric}
Let $(x\in X)=\frac{1}{r}(a_1,\dots,a_d)$ be a cyclic quotient singularity, then $$\mld(x,X)=\min_{1\leq j\leq r}\sum_{i=1}^d\left(1+\frac{ja_i}{r}-\lceil\frac{ja_i}{r}\rceil\right).$$
\end{lem}

\begin{defn}[Set of special cyclic quotient singularities]
We define several sets of $5$-dimensional cyclic quotient singularities for technical purposes. For every positive integer $r$, we let
\begin{itemize}
    \item $\mathcal{A}_r(1)$ be the set of all $5$-dimensional cyclic quotient singularities $(x\in X)$ of the form $\frac{1}{r}(a_1,a_2,a_3,a_4,a_5)$, such that
\begin{itemize}
    \item $a_1,a_2,a_3,a_4,a_5$ are positive integers,
    \item $a_i<r$ for each $i$, and
    \item $\mld(x,X)<2$,
\end{itemize}
\item $\mathcal{A}_r(2)$ the set of $(x\in X)=\frac{1}{r}(a_1,a_2,a_3,a_4,a_5)$ in $\mathcal{A}_r(1)$, such that
\begin{itemize}
    \item $\gcd(a_1,r)=\gcd(a_2,r)=\gcd(a_3,r)=1$, and
    \item $\gcd(a_4,r)=\gcd(a_5,r)$, 
\end{itemize}
\item $\mathcal{A}_r(3)$ the set of $(x\in X)=\frac{1}{r}(a_1,a_2,a_3,a_4,a_5)$ in $\mathcal{A}_r(2)$, such that $$\gcd\left(\sum_{i=1}^5a_i,r\right)=1,$$
\item $\mathcal{A}_r(4)$ the set of $(x\in X)=\frac{1}{r}(a_1,a_2,a_3,a_4,a_5)$ in $\mathcal{A}_r(3)$, such that one of the following holds:
    \begin{itemize}
        \item $r\mid a_1+a_4+a_5$, or
        \item $r\mid 2a_4+a_5$, or
        \item $r\mid 2a_1+a_5$ and $\gcd(a_4,r)\leq 2$,
\end{itemize}
and
\item $\mathcal{A}_r(5)$ the set of isolated cyclic quotient singularities $(x\in X)$ in $\mathcal{A}_r(1)$ (which is clear that $\mathcal{A}_r(5)\subset\mathcal{A}_r(2)$).
\end{itemize}
For every integer $i\in [1,5]$ and positive real number $\epsilon$, we define
\begin{itemize}
\item $\bar{\mathcal{A}}_r(i)$ to be the set of $(x\in X)=\frac{1}{r}(a_1,a_2,a_3,a_4,a_5)$ in $\mathcal{A}_r(i)$ such that $$\mld(x,X)=\frac{1}{r}\sum_{j=1}^5a_j,$$
\item $\mathcal{A}_r(i,\epsilon)$ to be the set of $(x\in X)\in\mathcal{A}_r(i)$ such that $\mld(x,X)>2-\epsilon$,
\item $\bar{\mathcal{A}}_r(i,\epsilon):=\mathcal{A}_r(i,\epsilon)\cap\bar{\mathcal{A}}_r(i)$, and
\item $$
\mathcal{A}(i)=\bigcup_{r=1}^{+\infty}\mathcal{A}_r(i),\bar{\mathcal{A}}(i)=\bigcup_{r=1}^{+\infty}\bar{\mathcal{A}}_r(i),
\mathcal{A}(i,\epsilon)=\bigcup_{r=1}^{+\infty}\mathcal{A}_r(i,\epsilon),
\bar{\mathcal{A}}(i,\epsilon)=\bigcup_{r=1}^{+\infty}\bar{\mathcal{A}}_r(i,\epsilon).$$
\end{itemize}
    We remark that if $\left\{a_1,a_2,a_3,a_4,a_5\right\}=\left\{a_1',a_2',a_3',a_4',a_5'\right\}$, we always identify $\frac{1}{r}(a_1,a_2,a_3,a_4,a_5)$ with $\frac{1}{r}(a_1',a_2',a_3',a_4',a_5')$ if they both belong to $\mathcal{A}(1)$.
\end{defn}

\begin{rem}
Theorem \ref{thm: main theorem special 5dimcyc} and is Theorem \ref{thm: isolated5 1/19} can be rephrased as $$\bar{\mathcal{A}}\left(4,\frac{1}{13}\right)=\left\{\frac{1}{19}(3,4,5,7,18),\frac{1}{17}(2,3,5,7,16),\frac{1}{14}(3,5,13,2,4)\right\}$$
and $$\mathcal{A}\left(5,\frac{1}{19}\right)=\emptyset$$ respectively.
\end{rem}

\subsection{Associated functions, conditions, and constants}

In this subsection, for any cyclic quotient singularity $(x\in X)=\frac{1}{r}(a_1,a_2,a_3,a_4,a_5)\in\mathcal{A}(1)$, we define an associated function $f(n)$ and study its properties. In particular, we introduce auxiliary equations $\mathcal{D}(n,c)$ and inequalities $\mathcal{C}(n)$ of $(x\in X)$ and study their properties.

\begin{defn}
 Let $q,c$ be two positive integers and $\bm{x}=(x_1,x_2,x_3,x_4,x_5)\in (0,1)^5$ a rational point. \begin{itemize}
     \item The \emph{associated function} of $\bm{x}$ is the function $f:\mathbb N^+\rightarrow\mathbb N$ given by $$f(n):=\sum_{i=1}^5\lfloor nx_i\rfloor$$
     \item If the minimal positive denominators of $x_4$ and $x_5$ both equal to $q$, then
     \begin{itemize}
    \item we say that \emph{condition $\mathcal{C}(n)$ holds} for $\bm{x}$ if 
    \begin{itemize}
        \item $n-1,n+1\in\Ii_q$, and
        \item $f(n-1)+5\leq f(n+1)$,
    \end{itemize}
    and
    \item we say that \emph{condition $\mathcal{D}(n,c)$ holds} for $\bm{x}$ if 
    \begin{itemize}
        \item $n\in\Ii_q$, and
        \item  $f(n)=2n-2-c$.
    \end{itemize}
\end{itemize}
 \end{itemize}
\end{defn}

\begin{defn}
Let $(x\in X)=\frac{1}{r}(a_1,a_2,a_3,a_4,a_5)\in\mathcal{A}(1)$ be a cyclic quotient singularity. 
\begin{itemize}
    \item The \emph{associated point} of $(x\in X)$ is defined as the point $$\bm{x}:=\left(\frac{a_1}{r},\frac{a_2}{r},\frac{a_3}{r},\frac{a_4}{r},\frac{a_5}{r}\right)\in (0,1)^5$$
    \item  The \emph{associated function} $f:\mathbb N^+\rightarrow\mathbb N$ of $(x\in X)$ is defined as the associated function of $\bm{x}$.
    \item If $(x\in X)\in\mathcal{A}(2)$, then
    \begin{itemize}
        \item the \emph{associated denominator} of $(x\in X)$ is defined as 
        $$q(x\in X):=\frac{r}{\gcd(a_4,r)}.$$
        \item we say that \emph{condition $\mathcal{C}(n)$ holds} for $(x\in X)$ if condition $\mathcal{C}(n)$ holds for $\bm{x}$, and
         \item we say that \emph{condition $\mathcal{D}(n,c)$ holds} for $(x\in X)$ if condition $\mathcal{D}(n,c)$ holds for $\bm{x}$.
    \end{itemize}
\end{itemize}
\end{defn}

The next lemma is the key lemma of our proof, which describes the behavior of $f(n)$ when $\mld(x,X)$ is sufficiently close to $2$.

\begin{lem}\label{lem: main lemma using trick}
Let $c$ be a positive integer and $\epsilon$ a positive real number. Let $(x\in X)=\frac{1}{r}(a_1,\dots, a_5)\in\bar{\mathcal{A}}(2)$ be a cyclic quotient singularity such that $\mld(x,X)=2-\epsilon$. Let $f:\mathbb N^+\rightarrow\mathbb N$ the associated function of $(x\in X)$ and $q:=q(x\in X)$ the associated denominator of $(x\in X)$. Then for every integer $n\in [2,r-1]\cap\Ii_q$,
 $$2n-3-(n+1)\epsilon\leq f(n)\leq2n-2-(n-1)\epsilon.$$
\end{lem}
\begin{proof}
We define $x_i:=\frac{a_i}{r}$ for each $i$. By Lemma \ref{lem: mld of toric}, for any integer $n\in [1,r-1]$,
$$\sum_{i=1}^5(1+nx_i-\lceil nx_i\rceil)\geq\mld(x,X)=\sum_{i=1}^5x_i.$$
For any integer $n\in [1,r-1]\cap\Ii_q$, since $$\left\{nx_i\right\}=1+nx_i-\lceil nx_i\rceil,$$ 
we have
$$\sum_{i=1}^5\left\{nx_i\right\}\geq\sum_{i=1}^5x_i,$$
which implies that
$$\sum_{i=1}^5(n-1)x_i\geq\sum_{i=1}^5\lfloor nx_i\rfloor=f(n).$$
Since for any $n\in [1,r-1]\cap\Ii_q$, $r-n\in [1,r-1]\cap\Ii_q$, we have
$$\sum_{i=1}^5(r-n-1)x_i\geq f(r-n).$$
Since $(x\in X)\in\bar{\mathcal{A}}(2)$, by the construction of $\bar{\mathcal{A}}(2)$, 
$$2-\epsilon=\mld(x,X)=\sum_{i=1}^5x_i.$$ 
By Lemma \ref{lem: round down plus round up}, 
\begin{align*}
2n-2-(n-1)\epsilon&=\sum_{i=1}^5(n-1)x_i\geq f(n)=\sum_{i=1}^5(rx_i-\lceil (r-n)x_i\rceil)\\
&=\sum_{i=1}^5rx_i-\sum_{i=1}^5(f(r-n)+1)\\
&\geq\sum_{i=1}^5rx_i-\sum_{i=1}^5(r-n-1)x_i-5=\sum_{i=1}^5(n+1)x_i-5\\
&=(n+1)(2-\epsilon)-5=2n-3-(n+1)\epsilon.
\end{align*}
\end{proof}

We apply Lemma \ref{lem: main lemma using trick} to study the behavior of the equations $\mathcal{D}(n,c)$ and the inequalities $\mathcal{C}(n)$. The necessary results in the rest of the proof are listed out in the following lemma:

\begin{lem}\label{lem: main lemma delta odd}
Let $c$ be a positive integer and $\epsilon$ a positive real number. Let $(x\in X)=\frac{1}{r}(a_1,\dots, a_5)\in\bar{\mathcal{A}}(2)$ be a cyclic quotient singularity such that $\mld(x,X)=2-\epsilon$. Let $f:\mathbb N^+\rightarrow\mathbb N$ the associated function of $(x\in X)$ and $q:=q(x\in X)$ the associated denominator of $(x\in X)$. Then for every integer $n\in [2,r-1]\cap\Ii_q$, we have the following:

\begin{enumerate}
 \item  If $n<\frac{c}{\epsilon}-1$, then $f(n)\geq 2n-2-c.$
 \item If $\frac{c-1}{\epsilon}+1<n<\frac{c}{\epsilon}-1$, then $\mathcal{D}(n,c)$ holds.
\item If $\mathcal{D}(n,c)$ holds, then
\begin{enumerate}
\item $\frac{c-1}{n+1}\leq\epsilon\leq\frac{c}{n-1}$,
\item for every $m\in [2,r-1]\cap\Ii_q$ such that $m<\frac{c+1}{c}\cdot(n-1)-1$, we have
$f(m)\geq 2m-3-c,$ and
\item if $D(n',c+1)$ holds for some integer $n'\in [2,r-1]\cap\Ii_q$, then $\mathcal{D}(m,c+1)$ holds for every integer $m\in [2,r-1]\cap\Ii_q$ such that 
$n'+3\leq m<\frac{c+1}{c}\cdot(n-1)-1$.
\end{enumerate}
\item If $n+1<r$, then
\begin{enumerate}
    \item if $\mathcal{D}(n-1,c+1)$ and $\mathcal{D}(n+1,c)$ hold, then $\epsilon=\frac{c}{n}$, 
    \item if $\mathcal{C}(n)$ holds, then $n\epsilon\in\mathbb N^+$, and
    \item if $(x\in X)\in\bar{\mathcal{A}}(3)$, then $\mathcal{C}(n)$ does not hold.
\end{enumerate}
\end{enumerate}
\end{lem}

\begin{proof} We prove the lemma part by part.

\begin{proof}[Proof of Lemma \ref{lem: main lemma delta odd}(1)] Since $n<\frac{c}{\epsilon}-1$, $$2n-3-(n+1)\epsilon>2n-3-c.$$
By Lemma \ref{lem: main lemma using trick},
$$2n-3-c<2n-3-(n+1)\epsilon\leq f(n).$$
Since $f(n)$ is an integer, $f(n)\geq 2n-2-c$.
\end{proof}

\begin{proof}[Proof of Lemma \ref{lem: main lemma delta odd}(2)]
By Lemma \ref{lem: main lemma using trick},
$$f(n)\leq 2n-2-(n-1)\epsilon<2n-2-\left(\left(\frac{c-1}{\epsilon}+1\right)-1\right)\epsilon=2n-1-c.$$
Since $f(n)$ is an integer, $f(n)\leq 2n-2-c$. By (1), $f(n)\geq 2n-2-c$. Thus $f(n)=2n-2-c$, which is equal to say that $\mathcal{D}(n,c)$ holds.
\end{proof}

\begin{proof}[Proof of Lemma \ref{lem: main lemma delta odd}(3.a)]
 Since $\mathcal{D}(n,c)$ holds, $f(n)=2n-2-c$. By Lemma \ref{lem: main lemma using trick},
$$2n-3-(n+1)\epsilon\leq 2n-2-c\leq 2n-2-(n-1)\epsilon,$$
so
$$\frac{c-1}{n+1}\leq\epsilon\leq\frac{c}{n-1}.$$
\end{proof}

\begin{proof}[Proof of Lemma \ref{lem: main lemma delta odd}(3.b)]
By (3.a), $\epsilon\leq\frac{c}{n-1}$, so 
$$m<\frac{c+1}{c}\cdot (n-1)-1\leq\frac{c+1}{\epsilon}-1.$$
By (1), $f(m)\geq 2m-2-(c+1)=2m-3-c$.
\end{proof}

\begin{proof}[Proof of Lemma \ref{lem: main lemma delta odd}(3.c)]
Apply (3.a) for $\mathcal{D}(n',c+1)$, we have 
$$\frac{c}{n'+1}\leq\epsilon$$
Thus
$$\frac{c}{\epsilon}+1\leq\frac{c}{\frac{c}{n'+1}}+1=n'+2.$$
Apply (3.a) for $\mathcal{D}(n,c)$, we have $\epsilon\leq\frac{c}{n-1}$, hence
$$\frac{c+1}{\epsilon}-1\geq\frac{c+1}{c}\cdot (n-1)-1.$$
(3.c) follows from (2).
\end{proof}

\begin{proof}[Proof of Lemma \ref{lem: main lemma delta odd}(4.a)]
Apply (3.a) for $\mathcal{D}(n-1,c+1)$, we have $\epsilon\geq\frac{c}{n}$. Apply (3.a) for $\mathcal{D}(n+1,c)$, we have $\epsilon\leq\frac{c}{n}$. Thus $\epsilon=\frac{c}{n}$.
\end{proof}

\begin{proof}[Proof of Lemma \ref{lem: main lemma delta odd}(4.b)]
By Lemma \ref{lem: main lemma using trick}, there exist positive integers $c_1,c_2$, such that $$f(n+1)=2(n+1)-2-c_1$$ and $$f(n-1)=2(n-1)-2-c_2.$$
Since $\mathcal{C}(n)$ holds, $$f(n-1)+5\leq f(n+1),$$
which implies that $c_1\leq c_2-1$.

Apply (3.a) for $\mathcal{D}(n+1,c_1)$, we have $$\epsilon\leq\frac{c_1}{(n+1)-1}=\frac{c_1}{n}.$$
Apply (3.a) for $\mathcal{D}(n-1,c_2)$, we have $$\epsilon\geq\frac{c_2-1}{(n-1)+1}=\frac{c_2-1}{n}.$$
Thus $c_2-1\leq c_1$, which implies that $c_2-1=c_1$. Thus $\mathcal{D}(n-1,c_1+1)$ and $\mathcal{D}(n+1,c_1)$ hold. By (4.a), $n\epsilon\in\mathbb N^+$.
\end{proof}

\begin{proof}[Proof of Lemma \ref{lem: main lemma delta odd}(4.c)]
Since $(x\in X)\in\bar{\mathcal{A}}(3)$,
$$\gcd\left(r\cdot\mld(x,X),r\right)=\gcd\left(2r-\sum_{i=1}^5a_i,r\right)=\gcd\left(\sum_{i=1}^5a_i,r\right)=1.$$
Thus $r$ is the minimal positive denominator of $\mld(x,X)$. Since $3\leq n+1<r$, $n\mld(x,X)\not\in\mathbb N^+$. Since $\mld(x,X)=2-\epsilon$, $n\epsilon\not\in\mathbb N^+$. By (4.b), $\mathcal{C}(n)$ does not hold.
\end{proof}

\end{proof}

\begin{rem}
We say a few words for Lemma \ref{lem: main lemma delta odd} to give the readers some intuition and how we apply emma \ref{lem: main lemma delta odd} to the proof. In this paper, when we apply Lemma \ref{lem: main lemma delta odd}, we usually take $c=1$ or $2$, and may take $c=3$ or $4$ only under some very special cases. Indeed, when $c=1$ or $2$, we get the most information. For example, assume that $r$ is sufficiently large, $q=q(x\in X)=r$ and take $c=1$. Then by Lemma \ref{lem: main lemma delta odd}(2), when $n\in\left(1,\frac{1}{\epsilon}-1\right)$, $\mathcal{D}(n,1)$ holds. That is,
$$\sum_{i=1}^5\lfloor nx_i\rfloor=2n-3$$
for every integer $n\in\left(1,\frac{1}{\epsilon}-1\right)$. 

As we usually assume that $\epsilon<\frac{1}{13}$, the equations above are expected to hold for every integer $n\in [2,12]$. Therefore, we get at least $12$ restriction equations for $$(x_1,x_2,x_3,x_4,x_5):=\left(\frac{a_1}{r},\frac{a_2}{r},\frac{a_3}{r},\frac{a_4}{r},\frac{a_5}{r}\right)\in (0,1)^5.$$ Similarly, we have
$$\sum_{i=1}^5\lfloor nx_i\rfloor=2n-4$$
for every $n\in \left(\frac{1}{\epsilon}+1,\frac{2}{\epsilon}-1\right)$. We get at least another $11$ equations for $(x_1,x_2,x_3,x_4,x_5)$. By enumerating the solutions of these equations by computer algorithm, we could restrict $(x_1,x_2,x_3,x_4,x_5)$ to a very small range. For all remaining $(x_1,x_2,x_3,x_4,x_5)$, we exclude them by applying Lemma \ref{lem: main lemma delta odd}(4)(5) and calculate by hand, and eventually get a contradiction. 

In the general case, the algorithm will become more complicated when $r$ or $q$ is small because we need to assume that $n\in [2,r-1]\cap\Ii_q$. Nevertheless, after some extra efforts, at the end of the day an enumeration will work out.
\end{rem}

\section{Theorems for calculation}
In this section we list out some theorems we have proved by computer algorithm. We write these theorems into forms that may be applied directly. For corresponding algorithm statements and the proof of these theorems, we refer the readers to the appendix.

\begin{thm}\label{thm: A41349}
$$\bigcup_{r=1}^{51}\bar{\mathcal{A}}_r\left(4,\frac{1}{13}\right)=\left\{\frac{1}{19}(3,4,5,7,18),\frac{1}{17}(2,3,5,7,16),\frac{1}{14}(3,5,13,2,4)\right\}.$$
\end{thm}

\begin{thm}\label{thm: dn1218}
For every $(x\in X)\in\bar{\mathcal{A}}(2)$ such that $q(x\in X)\geq 19$, $\mathcal{D}(n,1)$ does not hold for some integer $n\in [2,18]$.
\end{thm}

\begin{thm}\label{thm: d314}
For every $(x\in X)\in\bar{\mathcal{A}}(2)$, if
\begin{itemize}
    \item $q(x\in X)\geq 33$, and
    \item there exists an integer $k\in [13,18]$, such that
    \begin{itemize}
        \item $\mathcal{D}(n,1)$ holds for every integer $n\in [2,k-1]$,
        \item  $\mathcal{D}(n,2)$ holds for every integer $$n\in \{k\}\cup[k+2,\max\left\{2k-6,25\right\}],$$
    \end{itemize}
\end{itemize}
then  $k=16$ and $\mathcal{D}(31,4)$ holds.

\end{thm}

\begin{thm}\label{thm: dn1332}
For every $(x\in X)\in\bar{\mathcal{A}}(2)$ such that $3\leq q:=q(x\in X)\leq 32$, $\mathcal{D}(n,1)$ does not hold for some integer $n\in [2,28]\cap\Ii_q$.
\end{thm}

\begin{thm}\label{thm: d314d283}
For every $(x\in X)\in\bar{\mathcal{A}}(2)$, if
\begin{itemize}
    \item $3\leq q:=q(x\in X)\leq 32$, and
    \item there exists an integer $k\in [13,28]$, such that
    \begin{itemize}
        \item $\mathcal{D}(n,1)$ holds for every integer $n\in [2,k-1]\cap\Ii_q$, and
        \item $\mathcal{D}(n,2)$ holds for every integer $$n\in\left(\left\{k\right\}\cup [k+2,\max\left\{2k-8,25\right\}]\right)\cap\Ii_q,$$
    \end{itemize}
\end{itemize}
then one of the following holds:
\begin{enumerate}
    \item $\mathcal{C}(n)$ holds for some integer $n\in [2,45]$.
    \item $\mathcal{D}(31,4)$ holds.
    \item $\mathcal{D}(28,3)$ holds, $k=17$, and $q(x\in X)=3$.
    \item $\mathcal{D}(31,3)$ and $\mathcal{D}(35,2)$ hold.
\end{enumerate}
\end{thm}

\begin{thm}\label{thm: 3dimcyc13}
For every $(x\in X)=\frac{1}{r}(a_1,a_2,a_3,a_4,a_5)\in\bar{\mathcal{A}}_r(2)$ such that $r\geq 13$, $a_4=1$, and $a_5=r-1$, there exists an integer  $n\in [2,12]$ such that $\mathcal{D}(n,1)$ does not hold.
\end{thm}

We prove the following theorem by hand, as algorithm is more complicated. The proof is also in the appendix. 
\begin{thm}\label{thm: d213}
For every $(x\in X)\in\bar{\mathcal{A}}(4)$ such that $q(x\in X)=2$ and $\mld(x,X)>2-\frac{1}{4}$, there exists $n\in\left\{3,5,7,9\right\}$ such that $\mathcal{D}(n,1)$ does not hold.
\end{thm}

\section{Proof of Theorem \ref{thm: main theorem special 5dimcyc}}

\begin{proof}[Proof of Theorem \ref{thm: main theorem special 5dimcyc}] Since $\gcd(a_1,r)=\gcd(a_2,r)=\gcd(a_3,r)=1$ and $a_1,a_2,a_3\in [1,r]$, $a_i<r$ for every $i\in\{1,2,3\}$. Since $\gcd(a_4,r)=\gcd(a_5,r)$ and $a_4,a_5\in [1,r]$, either $a_4=a_5=r$ or $a_4<r, a_5<r$. Since $$2>\frac{1}{r}\sum_{i=1}^5a_i>\frac{1}{r}(a_4+a_5),$$ 
$a_4<r$ and $a_5<r$. Thus $a_i<r$ for every $i$. By definition of $\bar{\mathcal{A}}\left(4,\frac{1}{13}\right)$ and our assumptions, $(x\in X)\in\bar{\mathcal{A}}\left(4,\frac{1}{13}\right)$.

If $r\leq 51$, then the theorem follows from Theorem \ref{thm: A41349}. Thus we may assume that $r\geq 52$ in the rest of the proof. 

Let $\epsilon:=2-\mld(x,X)$ and $q:=q(x\in X)$. Then $\epsilon<\frac{1}{13}$ and $q\geq 2$. Apply Lemma \ref{lem: main lemma delta odd}(2) for $\epsilon$ and $c=1$, we deduce that $\mathcal{D}(n,1)$ holds for every integer $n\in [2,12]\cap\Ii_q$. If $q=2$, we get a contradiction to Theorem \ref{thm: d213}. Thus we may assume that $q\geq 3$ in the rest of the proof. 

Since $r\cdot\mld(x,X)\in\mathbb N^+$, $r\epsilon\in\mathbb N^+$.

\begin{claim}\label{claim: repsilon not 1}
$r\epsilon\geq 2$.
\end{claim}
\begin{proof}[Proof of Claim \ref{claim: repsilon not 1}]
Suppose not. Then $r\epsilon=1$. By  Lemma \ref{lem: main lemma delta odd}(2) for $\epsilon=\frac{1}{r}$ and $c=1$,  $\mathcal{D}(n,1)$ holds for every integer $n\in [2,r-2]\cap\Ii_q$. Since $r\geq 52$, $\mathcal{D}(n,1)$ holds for every integer $n\in [2,50]\cap\Ii_q$. Since $q\geq 3$, we get a contradiction to Theorem \ref{thm: dn1218} and Theorem \ref{thm: dn1332}.
\end{proof}

Since $r\epsilon\geq 2$ and $\epsilon<\frac{1}{13}$, 
$$\left(\frac{1}{\epsilon}+1,\frac{2}{\epsilon}-1\right)\cap\Ii_q\not=\emptyset.$$
By Lemma \ref{lem: main lemma delta odd}(2), $\mathcal{D}(n,2)$ holds for some positive integer $n\geq 2$. Thus we may define
$$k:=\min\left\{n\geq 2|\mathcal{D}(n,2)\text{ holds}\right\}.$$
Since $\mathcal{D}(n,1)$ holds for every integer $n\in [2,12]\cap\Ii_q$, $k\geq 13$. We have the following.
\begin{claim}\label{claim: dn1 for 2 k-1}
For every integer $n\in\Ii_q\cap [2,k-1]$, $\mathcal{D}(n,1)$ holds.
\end{claim}
\begin{proof}[Proof of Claim \ref{claim: dn1 for 2 k-1}]
We let 
$$k':=\min\left\{n\geq 2|\mathcal{D}(n,c)\text{ holds for some }c\geq 2\right\}.$$
Then $k'\leq k$. If $k'=k$ then we are done. Otherwise, $k'<k$, and $\mathcal{D}(k',c)$ holds for some $c\geq 3$. By Lemma \ref{lem: main lemma delta odd}(3.a), $$\epsilon\geq\frac{c-1}{k'+1}\geq\frac{2}{k'+1}.$$ 
On the other hand, since $\mathcal{D}(n,1)$ holds for every integer $n\in [2,12]\cap\Ii_q$, $k'\geq 13$. Thus either $k'-2\in\Ii_q$ or $k'-1\in\Ii_q$. By the construction of $k'$ and Lemma \ref{lem: main lemma using trick}, either $\mathcal{D}(k'-2,1)$ holds our $\mathcal{D}(k'-1,1)$ holds. By Lemma \ref{lem: main lemma delta odd}(3.a), $$\epsilon\leq\max\left\{\frac{1}{(k'-1)-1},\frac{1}{(k'-2)-1}\right\}=\frac{1}{k'-2}.$$
Since $k'\geq 13$, 
$$\epsilon\leq\frac{1}{k'-2}<\frac{2}{k'+1}\leq\epsilon,$$ 
a contradiction.
\end{proof}

\begin{claim}\label{claim: cn do not hold}
$\mathcal{C}(n)$ does not hold for every integer $n\in [2,50]$. 
\end{claim}
\begin{proof}[Proof of Claim \ref{claim: cn do not hold}]
Since $(x\in X)\in\bar{\mathcal{A}}(4)\subset\bar{\mathcal{A}}(3)$ and $r\geq 52$, Claim \ref{claim: cn do not hold} follows immediately from Lemma \ref{lem: main lemma delta odd}(4.c).
\end{proof}

\begin{claim}\label{claim: d314}
$\mathcal{D}(31,4)$ does not hold. 
\end{claim}
\begin{proof}[Proof of Claim \ref{claim: d314}]
If $\mathcal{D}(31,4)$ holds, then by Lemma \ref{lem: main lemma delta odd}(3.a), $\epsilon\geq\frac{3}{32}>\frac{1}{13}$, a contradiction.
\end{proof}

\begin{claim}\label{claim: d352 and d313}
$\mathcal{D}(31,3)$ and $\mathcal{D}(35,2)$ do not hold together.
\end{claim}
\begin{proof}[Proof of Claim \ref{claim: d314}]
If $\mathcal{D}(31,3)$ holds, then by Lemma \ref{lem: main lemma delta odd}(4.a), $\frac{1}{16}\leq\epsilon$. If $\mathcal{D}(35,2)$ holds, then by Lemma \ref{lem: main lemma delta odd}(4.a), $\epsilon\leq\frac{1}{17}$. So $\mathcal{D}(31,3)$ and $\mathcal{D}(35,2)$ do not hold together.
\end{proof}

\begin{claim}\label{claim: k+2 to max 25 2k-8}
Assume that $k\leq 28$. Then
\begin{enumerate}
    \item For every integer $n\in [k+2,\max\{25,2k-8\}]\cap\Ii_q$, $\mathcal{D}(n,2)$ holds.
    \item If $k-1\in\Ii_q$, then for every integer $n\in [k+2,\max\{25,2k-6\}]\cap\Ii_q$, $\mathcal{D}(n,2)$ holds.
\end{enumerate}
\end{claim}
\begin{proof}[Proof of Claim \ref{claim: k+2 to max 25 2k-8}]
Since $\epsilon<\frac{1}{13}$, by Lemma \ref{lem: main lemma delta odd}(1), for every integer $n\in [2,25]\cap\Ii_q$, $f(n)\geq 2n-4$. By Lemma \ref{lem: main lemma using trick}, for every integer $n\in[2,25]\cap\Ii_q$, either $\mathcal{D}(n,1)$ holds or $\mathcal{D}(k+2,2)$ holds.

By Lemma \ref{lem: main lemma delta odd}(4.c), $\mathcal{C}(k+1)$ does not hold. Since $\mathcal{D}(k,2)$ holds, $\mathcal{D}(k+2,1)$ does not hold. So if $k+2\in\Ii_q$, then $\mathcal{D}(n,2)$ holds.

Since $k\geq 13$, $k-2>2$. So either $k-2\in\Ii_q$ or $k-1\in\Ii_q$. There are two cases.

\medskip

\noindent\textbf{Case 1}. $k-1\in\Ii_q$. By Claim \ref{claim: dn1 for 2 k-1}, $\mathcal{D}(k-1,1)$ holds. Since $\mathcal{D}(k,2)$ holds, by Lemma \ref{lem: main lemma delta odd}(3.c), for every $n\in [2,r-1]\cap\Ii_q$ such that $k+3\leq n<2k-5$, $\mathcal{D}(n,2)$ holds. Since $r\geq 52$ and $k\leq 28$, for every $n\in [k+3,2k-6]\cap\Ii_q$, $\mathcal{D}(n,2)$ holds. Thus  for every $n\in [k+2,2k-6]\cap\Ii_q$, $\mathcal{D}(n,2)$ holds.

If $2k-6\geq 25$ then we are done. Otherwise, we may assume that there exists an integer $m\in [2k-5,25]$ such that $m\in\Ii_q$ and $\mathcal{D}(m,2)$ does not hold. In this case, $\mathcal{D}(m,1)$ holds. By Lemma \ref{lem: main lemma delta odd}(3.a), $\epsilon\leq\frac{1}{m-1}$. By Lemma \ref{lem: main lemma delta odd}(2), for every integer $n\in [2,m-3]\cap\Ii_q$, $\mathcal{D}(n,1)$ holds. Since $m\geq 2k-5$, for every integer $n\in [2,2k-8]\cap\Ii_q$, $\mathcal{D}(n,1)$ holds. But $\mathcal{D}(k,2)$ holds and $2\leq k\leq 2k-8$, a contradiction.

\medskip

\noindent\textbf{Case 2}. $k-2\in\Ii_q$. By Claim \ref{claim: dn1 for 2 k-1}, $\mathcal{D}(k-2,1)$ holds. Since $\mathcal{D}(k,2)$ holds, by Lemma \ref{lem: main lemma delta odd}(3.c), for every $n\in [2,r-1]\cap\Ii_q$ such that $k+3\leq n<2k-7$, $\mathcal{D}(n,2)$ holds. Since $r\geq 52$ and $k\leq 28$, for every $n\in [k+3,2k-8]\cap\Ii_q$, $\mathcal{D}(n,2)$ holds. Thus  for every $n\in [k+2,2k-8]\cap\Ii_q$, $\mathcal{D}(n,2)$ holds.

If $2k-8\geq 25$ then we are done. Otherwise, we may assume that there exists an integer $m\in [2k-7,25]$ such that $m\in\Ii_q$ and $\mathcal{D}(m,2)$ does not hold. In this case, $\mathcal{D}(m,1)$ holds. By Lemma \ref{lem: main lemma delta odd}(3.a), $\epsilon\leq\frac{1}{m-1}$. By Lemma \ref{lem: main lemma delta odd}(2), for every integer $n\in [2,m-3]\cap\Ii_q$, $\mathcal{D}(n,1)$ holds. Since $m\geq 2k-7$, for every integer $n\in [2,2k-10]\cap\Ii_q$, $\mathcal{D}(n,1)$ holds. But $\mathcal{D}(k,2)$ holds and $2\leq k\leq 2k-10$, a contradiction.
\end{proof}

\noindent\textit{Proof of Theorem \ref{thm: main theorem special 5dimcyc} continued}. There are four cases of $(k,q)\in\mathbb N^+\times\mathbb N^+$.
\begin{itemize}
\item[\textbf{Case 1}] $k>18$ and $q>32$. Then Claim \ref{claim: dn1 for 2 k-1} contradicts Theorem \ref{thm: dn1218}.
\item[\textbf{Case 2}] $13\leq k\leq 18$ and $q>32$. In this case, $$2,3,\dots,k-1,k,\dots,\max\{25,2k-6\}\in\Ii_q.$$ Then Claim \ref{claim: dn1 for 2 k-1}, Claim \ref{claim: k+2 to max 25 2k-8} and Claim \ref{claim: d314} contradict Theorem \ref{thm: d314}.
\item[\textbf{Case 3}] $k>28$ and $3\leq q\leq 32$. Then Claim \ref{claim: dn1 for 2 k-1} contradicts Theorem \ref{thm: dn1332}.
\item[\textbf{Case 4}] $13\leq k\leq 28$ and $3\leq q\leq 32$. In this case, by Claim \ref{claim: dn1 for 2 k-1}, Claim \ref{claim: k+2 to max 25 2k-8}, Claim \ref{claim: cn do not hold} Claim \ref{claim: d314}, Claim \ref{claim: d352 and d313}, and Theorem \ref{thm: d314d283}, $\mathcal{D}(28,3)$ holds, $k=17$, and $q(x\in X)=3$. Since  $q=3\nmid 16$, $16\in\Ii_q$. By  Claim \ref{claim: dn1 for 2 k-1}, $\mathcal{D}(16,1)$ holds. By Lemma \ref{lem: main lemma delta odd}(3.b), $\mathcal{D}(28,3)$ does not hold, a contradiction.
\end{itemize}
\end{proof}

\section{Proof of Theorem \ref{thm: isolated5 1/19}}
\begin{lem}\label{lem: reduce cyc sing}
For any integer positive integer $r$, integer $l\in [1,5]$, and $(x\in X)=\frac{1}{r}(a_1,a_2,a_3,a_4,a_5)\in\mathcal{A}_r(l)$, there exists an integer $r'\in (0,r]$ and $(x'\in X')\in\bar{\mathcal{A}}_{r'}(l)$, such that $\mld(x',X')=\mld(x,X)$. 
\end{lem}
\begin{proof}
By Lemma \ref{lem: mld of toric}, there exists an integer $j\in [1,r-1]$, such that
$$\mld(x,X)=\sum_{i=1}^5(1+jx_i-\lceil jx_i\rceil).$$
We let
$$r':=\frac{r}{\gcd(j,r)},a_i':=r'\left(1+\frac{ja_i}{r}-\lceil \frac{ja_i}{r}\rceil\right) \text{ for each } i,$$
and
$$(x'\in X'):=\frac{1}{r'}(a_1',a_2',a_3',a_4',a_5').$$

We show that $(x'\in X')$ satisfies the requirements. By construction, we only need to show that $\mld(x,X)=\mld(x',X')$, which follows from
\begin{align*}
\mld(x',X')&=\min\left\{\sum_{i=1}^5b_i|\exists n\in\mathbb N^+, \forall i, b_i>0, b_i\equiv\frac{na'_i}{r'}\ \Mod\ \Zz\right\}\\
&=\min\left\{\sum_{i=1}^5b_i|\exists n\in\mathbb N^+, \forall i, b_i>0, b_i\equiv\frac{nja_i}{r}\ \Mod\ \Zz\right\}\\
&\geq\min\left\{\sum_{i=1}^5b_i|\exists n\in\mathbb N^+, \forall i, b_i>0, b_i\equiv\frac{na_i}{r}\ \Mod\ \Zz\right\}\\
&=\mld(x,X)=\frac{1}{r'}\sum_{i=1}^5a_i'\geq\mld(x',X').
\end{align*}
\end{proof}

\begin{proof}[Proof of Theorem \ref{thm: isolated5 1/19}]
By Lemma \ref{lem: reduce cyc sing}, we may assume that $(x\in X)\in\bar{\mathcal{A}}(5)\subset\bar{\mathcal{A}}(2)$. By Lemma \ref{lem: main lemma delta odd}(2), $\mathcal{D}(n,1)$ holds for every integer $n\in [2,18]$, which contradicts Theorem \ref{thm: dn1218}.
\end{proof}

\section{Proof of Theorem \ref{thm: 12/13}}

In this section, we want to use Theorem \ref{thm: main theorem special 5dimcyc} to prove Theorem \ref{thm: 12/13}. Since we will eventually use \cite[Theorem 4.1]{Jia19} and \cite[Proof of Theorem 1.3]{Jia19}, we need to adopt the notions in \cite{Jia19}. Therefore, we need to define a new class of $5$-dimensional cyclic quotient singularities and a new constant $k_0(x\in X)$.

\begin{defn}
We let $\mathcal{B}_r$ be the set of all $5$-dimensional cyclic quotient singularities of the form $(x\in X)=\frac{1}{r}(a_1,a_2,a_3,a_4,-e)$ satisfying the following:
\begin{itemize}
    \item $\gcd(a_1,r)=\gcd(a_2,r)=\gcd(a_3,r)=\gcd(\sum_{i=1}^4a_i-e,r)=1$,
    \item $\gcd(a_4,r)=\gcd(e,r)$,
    \item there exists an integer $k_0:=k_0(x\in X)\in [1,r-1]$, such that
    \begin{itemize}
\item $r\nmid k_0e$, 
\item $$\sum_{i=1}^4\left\{\frac{k_0a_i}{r}\right\}-\left\{\frac{k_0e}{r}\right\}=\frac{k_0}{r},$$ and
\item  $$\sum_{i=1}^4\left\{\frac{ka_i}{r}\right\}-\left\{\frac{ke}{r}\right\}\geq 1$$ for every integer $k\in [1,r-1]$ such that $k\not=k_0$,
\end{itemize}
and
\item one of the following holds:
\begin{itemize}
\item $r\mid a_1+a_2-e$, or
\item $r\mid 2a_4-e$, or
\item $r\mid 2a_1-e$ and $\gcd(e,r)\leq 2$,
\end{itemize}
\end{itemize}

We define $\bar{\mathcal{B}}_r\subset\mathcal{B}_r$ to be the subset of all $(x\in X)\in\mathcal{B}_r$ such that 
$$\gcd\left(k_0(x\in X),r\right)=1.$$ 
We define $\mathcal{B}:=\cup_{r=1}^{+\infty}\mathcal{B}_r$ and $\bar{\mathcal{B}}:=\cup_{r=1}^{+\infty}\bar{\mathcal{B}}_r$.
\end{defn}

\begin{lem}\label{lem: classification six singularities}
For every $(x\in X)=\frac{1}{r}(a_1,a_2,a_3,a_4,-e)\in\bar{\mathcal{B}}$ such that $1+\frac{k_0}{r}>2-\frac{1}{13}$ where $k_0:=k_0(x\in X)$, one of the following cases holds:

\begin{enumerate}
\item $r=14, k_0=13, (\left\{a_1,a_2,a_3,a_4\right\},e)\equiv (\left\{1,9,11,10\right\},2) (\Mod\ r)$. 
\item $r=14, k_0=13, (\left\{a_1,a_2,a_3,a_4\right\},e)\equiv (\left\{9,1,11,12\right\},4) (\Mod\ r)$.
\item $r=17, k_0=16, (\left\{a_1,a_2,a_3,a_4\right\},e)\equiv (\left\{1,10,12,14\right\},2)  (\Mod\ r)$.
\item $r=17, k_0=16, (\left\{a_1,a_2,a_3,a_4\right\},e)\equiv (\left\{1,10,12,15\right\},3)  (\Mod\ r)$.
\item $r=17, k_0=16, (\left\{a_1,a_2,a_3,a_4\right\},e)\equiv (\left\{1,12,14,15\right\},7)  (\Mod\ r)$.
\item $r=19, k_0=18, (\left\{a_1,a_2,a_3,a_4\right\},e)\equiv (\left\{1,12,15,16\right\},5)  (\Mod\ r)$.
\end{enumerate}
\end{lem}
\begin{proof}
Since $\mld(x,X)=1+\frac{k_0}{r}$ and $\gcd(k_0,r)=1$, the Lemma follows from Theorem \ref{thm: main theorem special 5dimcyc}.
\end{proof}

\begin{lem}\label{lem: bad singularities not possible}
Let $\overline{M}\cong\Zz^{4}$ be the lattice of monomials on $\mathbb A^4$, $\overline{N}$ the dual of $\bar M$, $\sigma:=\Rr^4_{\geq 0}$ and $(x\in X)=\frac{1}{r}(a_1,a_2,a_3,a_4,-e)\in\bar{\mathcal{B}}$ a cyclic quotient singularity. We define $N:=\overline{N}+\Zz\cdot\frac{1}{r}(a_1,a_2,a_3,a_4)$. Suppose that $1+\frac{k_0(x\in X)}{r}>2-\frac{1}{13}$. Then there exists an integer $1\leq j\leq r-1$ and a weighting $\alpha_j=\frac{1}{r}(a_1(j),a_2(j),a_3(j),a_4(j))\in N\cap\sigma$, such that 
\begin{enumerate}
    \item $a_i(j)\equiv ja_i (\Mod\ r)$ for every $1\leq i\leq 4$, and
    \item for any monomial $\bm{x}=\prod_{i=1}^4x_i^{c_i}$ such that 
\begin{itemize}
    \item $(c_1,c_2,c_3,c_4)\in\mathbb N^4\backslash\left\{\bm{0}\right\}$, and
    \item $\sum_{i=1}^4c_ia_i\equiv e (\Mod\ r)$,
\end{itemize} 
we have $\alpha_j(x_1x_2x_3x_4)\leq 1-\frac{1}{13}+\alpha_j(\bm{x})$.
\end{enumerate}
\end{lem}
\begin{proof}
Let $k_0:=k_0(x\in X)$. By Lemma \ref{lem: classification six singularities}, one of the cases (1)--(6) of Lemma \ref{lem: classification six singularities} holds.  If we are in case (1) (resp. (2),(3),(4),(5),(6)) of Lemma \ref{lem: classification six singularities}, we do the following and finish the proof:
\begin{itemize}
    \item[\textbf{Step 1}] After possibly perturbing $(a_1,a_2,a_3,a_4)$, we may suppose that 
$$(a_1,a_2,a_3,a_4)\equiv (1,9,11,10)$$
(resp.
$$(9,1,11,12),(1,10,12,14),(1,10,12,15),(1,12,14,15),(1,12,15,16) (\Mod\ r)\big).$$
\item[\textbf{Step 2}] We let 
$$j:=8 \text{ (resp. } 4,9,6,4,4)$$ and $$\alpha_j:=\frac{1}{14}(8,2,10,4)$$ 
(resp. 
$$\frac{1}{14}(8,4,2,6), \frac{1}{17}(9,5,6,7), \frac{1}{14}(6,9,4,5), \frac{1}{14}(4,14,5,9), \frac{1}{19}(4,10,3,7)\big).$$
\item[\textbf{Step 3}] If $\alpha_j(x_1x_2x_3x_4)>1-\frac{1}{13}+\alpha_j(\bm{x})$ for some $\bm{x}=\sum_{i=1}^4x_i^{c_i}$ such that $(c_1,c_2,c_3,c_4)\in\mathbb N^4\backslash\left\{\bm{0}\right\}$, then an easy calculation gives 
$$\sum_{i=1}^4c_i\alpha_i(j)\leq 11\text{ (resp.}\leq 7,11,8,16,6).$$
\item[\textbf{Step 4}] After enumerating all possibilities of $(c_1,c_2,c_3,c_4)$, we find that 
$$\sum_{i=1}^4c_ia_i\not\equiv e (\Mod\ r),$$ 
which contradicts our assumptions.
\end{itemize}
\end{proof}

The next lemma can be calculated by hand, but it indeed follows from Theorem \ref{thm: 3dimcyc13}.
\begin{lem}\label{lem: 3dimisolatedcyc1213}
Let $(x\in X)$ be a $3$-dimensional isolated cyclic quotient singularity such that $\mld(x,X)<1$, then $\mld(x,X)\leq\frac{12}{13}$.
\end{lem}
\begin{proof}
We may assume that $(x\in X)=\frac{1}{r}(a_1,a_2,a_3)$ for some integers $0<a_1,a_2,a_3<r$, such that $\mld(x,X)=\sum_{i=1}^3\left\{\frac{ja_i}{r}\right\}$ for some integer $0<j<r$. Possibly replacing $r$ with $\frac{r}{\gcd(j,r)}$ and each $a_i$ with $\frac{r}{\gcd(j,r)}\cdot\left\{\frac{ja_i}{r}\right\}$, we may assume that $\mld(x,X)=\sum_{i=1}^3\frac{a_i}{r}$. In particular, since $(x\in X)$ is an isolated cyclic quotient singularity, $\gcd(a_1,r)=\gcd(a_2,r)=\gcd(a_3,r)=1$.

Thus $(y\in Y):=\frac{1}{r}(a_1,a_2,a_3,a_4:=1,a_5:=r-1)$ is a cyclic quotient singularity such that $(y\in Y)\in\bar{\mathcal{A}}_r(2)$. By Theorem \ref{thm: 3dimcyc13}, $\mathcal{D}(n,1)$ does not hold for some $2\leq n\leq 12$. By Lemma \ref{lem: main lemma delta odd}(3), $2-\mld(y,Y)\geq\frac{1}{13}$, which implies that $\mld(x,X)\leq\frac{12}{13}$.
\end{proof}

\begin{proof}[Proof of Theorem \ref{thm: 12/13}]
By \cite[Theorem 4.1]{Jia19} and \cite[Proof of Theorem 1.3]{Jia19}, we may assume that there exist integers $a_1,a_2,a_3,a_4,e,r>0$ and a function $f=f(x,y,z,t)$, such that $(x\in X)$ is a hyperquotient singularity of the form $(x\in X)=\left(y\in Y:=\left\{f=0\right\}\in\mathbb A^4\right)/\bm{\mu}_r$ satisfying the following.
\begin{itemize}
    \item $(y\in Y)$ is an isolated $cDV$ singularity,
    \item $(y\in Y)$ is a canonical $1$-cover of $(x\in X)$, 
    \item there exists exactly $1$ prime divisor $E$ over $(x\in X)$, such that $a(E,X)\leq 1$, and
    \item $\bm{\mu}_r$ is the cyclic group of order $r$ acting on $(y\in Y)$ in the following way:
    $$\bm{\mu}_r\ni\xi:(x_1,x_2,x_3,x_4;f)\rightarrow \left(\xi^{a_1}x_1,\xi^{a_2}x_2,\xi^{a_3}x_3,\xi^{a_4}x_4;\xi^{e}f\right),$$
    where $\xi$ is the primitive root of order $r>0$.
\end{itemize}
By \cite[Rule I,II,III, Propostion 4.3, Propositon 4.4]{Jia19} and Lemma \ref{lem: 3dimisolatedcyc1213}, we may assume that $(z\in Z):=\frac{1}{r}(a_1,a_2,a_3,a_4,-e)\in\mathcal{B}$. If $\mld(z,Z)\leq 2-\frac{1}{13}$, then the theorem follows from \cite[Remark 2.13]{Jia19}. If $\mld(z,Z)>2-\frac{1}{13}$, we let
$$r':=\frac{r}{\gcd(k_0(z\in Z),r)}$$
and
$$(x'\in X'):=\frac{1}{r'}(a_1,a_2,a_3,a_4,-e).$$
Then $(x'\in X')\in\mathcal{B}_{r'}$ and 
$$k_0':=k_0(x'\in X')=\frac{k_0}{\gcd(k_0,r)}.$$
So $\gcd(k_0',r)=1$, which implies that $(x'\in X')\in\mathcal{B}'_{r'}$. Apply Lemma \ref{lem: bad singularities not possible} for $(x'\in X')$, and we get a contradiction to \cite[Rule I]{Jia19}.
\end{proof}

\noindent{\textbf{Acknowledgement}}. The first author would like to have a very special thank to Chen Jiang, who not only kindly shared the draft of his preprint \cite{Jia19} which is crucial to this paper, but also gave the author some essential ideas of Lemma \ref{lem: main lemma delta odd} and Lemma \ref{lem: bad singularities not possible}. He would like to thank Jiaming Li for her important help in the proof of Theorem \ref{thm: A41349}. Some intuitions of this work were provided by James M\textsuperscript{c}Kernan when the first author attended the 2018 Fall Program of Moduli Spaces and Varieties in SCMS. He would like to thank James M\textsuperscript{c}Kernan for the useful discussion and SCMS for their hospitality. Part of the work was done when the first author made an unofficial visit to the Mathematical Sciences Research Institute in Berkeley, California during the Spring 2019 semester, and he would like to thank his advisor Christopher D. Hacon for his significant support and warm encouragement. He would also like to thank Jingjun Han, Yuchen Liu, Vyacheslav V. Shokurov, Chenyang Xu, Chuyu Zhou and Ziquan Zhuang for discussions and comments during the preparation of this paper. The first author was partially supported by NSF research grants no: DMS-1801851, DMS-1265285 and by a grant from the Simons Foundation; Award Number: 256202.

The second author would like to thank her manager Aleksey Sanin and VP Peeyush Ranjan in Google for their effort to read through this article, make sure this work does not have interest conflicts with Google, and allow her to finish this work in her spare time.

The authors would like to thank the referees for carefully checking the details and many useful suggestions.

\appendix
\section{Proof of theorems in Section 3}

In this appendix we state the theorems we get by computer programs, each corresponding to one theorem of Section 3. As most of our algorithms are similar to the algorithm implementing Theorem \ref{thm: 17equations no solution}, we will explain this algorithm, and refer the readers to the first version of this paper on arXiv \cite[Appendix B]{LX19} for the precise algorithms.

\begin{thm}\label{thm: 5dimcycrleq49}
Find all integers $r,a_1,a_2,a_3,a_4,a_5$, such that

$14\leq r\leq 51$ and $0<a_1\leq a_2\leq a_3<r$ and $0<a_4\leq a_5<r$ such that
\begin{itemize}
\item $14\leq r\leq 51$,
\item $0<a_1\leq a_2\leq a_3<r$,
\item $0<a_4\leq a_5<r$,
\item $a_1+a_2+a_3+a_4+a_5=2r-1$ or $2r-2$ or $2r-3$,
\item $\gcd(a_1,r)=\gcd(a_2,r)=\gcd(a_3,r)=1$, 
\item $\gcd(a_4,r)=\gcd(a_5,r)$, and
\item for every integer $n\in [1,r-1]$
 $$\sum_{i=1}^5\left(1+\frac{na_i}{r}-\lceil\frac{na_i}{r}\rceil\right)\geq\frac{1}{r}\sum_{i=1}^5a_i>2-\frac{1}{13}.$$
\end{itemize}
Then we have
\begin{enumerate}
\item $r=17$ and $\left\{a_1,a_2,a_3,a_4,a_5\right\}=\left\{2,3,5,7,16\right\}$, or
\item $r=19$ and $\left\{a_1,a_2,a_3,a_4,a_5\right\}=\left\{3,4,5,7,18\right\}$, or
\item 
\begin{align*}
\frac{1}{r}(a_1,a_2,a_3,a_4,a_5)\in\big\{&\frac{1}{14}(3, 5, 13, 2, 4),\frac{1}{16}(1, 5, 9, 8, 8),\frac{1}{25}(6, 7, 11, 5, 20),\\
&\frac{1}{25}(6, 7, 11, 10, 15), \frac{1}{27}(4, 10, 11, 9, 18),\\
&\frac{1}{32}(5, 9, 17, 16, 16),\frac{1}{33}(8, 10, 13, 11, 22),\\
&\frac{1}{44}(7, 9, 25, 22, 22)\big\}.
\end{align*}
\end{enumerate}
\end{thm}

\begin{thm}\label{thm: 17equations no solution} Find all solutions of real numbers $x_1,x_2,x_3,x_4,x_5$, such that
\begin{itemize}
    \item $$0\leq x_1\leq x_2\leq x_3\leq x_4\leq x_5<1,$$
    and
    \item the equations
    $$\sum_{i=1}^5\lfloor nx_i\rfloor=2n-3$$
    hold for every integer $n\in [2,18]$.
\end{itemize}
Then we have no solution.
\end{thm}

\begin{thm}\label{thm: k1318noq}
Find all integers $k$ and real numbers $x_1,x_2,x_3,x_4,x_5$, such that
\begin{itemize}
\item $13\leq k\leq 18$,
\item $$0<x_1\leq x_2\leq x_3\leq x_4\leq x_5<1,$$
\item $$\sum_{i=1}^5\lfloor nx_i\rfloor=2n-3$$ 
for every integer $n\in [2,k-1]$, and
\item $$\sum_{i=1}^5\lfloor nx_i\rfloor=2n-4$$ 
for every integer $n$ such that 
\begin{itemize}
    \item $k\leq n\leq\max\left\{2k-6,25\right\}$, and
    \item $n\not=k+1$.
\end{itemize}
\end{itemize}
Then $k=16$, $$x_1\in\left(\frac{2}{13},\frac{3}{19}\right), x_2\in\left(\frac{5}{23},\frac{2}{9}\right), x_3\in\left(\frac{7}{25},\frac{2}{7}\right), x_4\in\left(\frac{1}{3},\frac{9}{26}\right), x_5\in\left(\frac{13}{14},\frac{14}{15}\right).$$
\end{thm}

\begin{thm}\label{thm: q332nok} Find all integers $q,b,c$ and real numbers $x_1,x_2,x_3$, such that
\begin{itemize}
    \item $3\leq q\leq 32$,
    \item $0<x_1\leq x_2\leq x_3<1$,
    \item $0<b\leq c<q$,
    \item $\gcd(b,q)=\gcd(c,q)=1$, and
    \item $$\sum_{i=1}^3\lfloor nx_i\rfloor+\lfloor\frac{nb}{q}\rfloor+\lfloor\frac{nc}{q}\rfloor=2n-3$$ 
    for every integer $n$ such that
    \begin{itemize}
        \item $q\nmid n$, and
        \item $2\leq n\leq 28$.
    \end{itemize}
\end{itemize}
Then we have no solution.
\end{thm}

\begin{thm}\label{thm: 17equations no solution divi6}
Find all integers $k,q,b,c$ and real numbers $x_1,x_2,x_3$, such that
\begin{itemize}
    \item $13\leq k\leq 28$,
    \item $3\leq q\leq 32$,
    \item $0<x_1\leq x_2\leq x_3<1$,
    \item $0\leq b\leq c<q$,
    \item $\gcd(b,q)=\gcd(c,q)=1$,
    \item $$\sum_{i=1}^3\lfloor nx_i\rfloor+\lfloor\frac{nb}{q}\rfloor+\lfloor\frac{nc}{q}\rfloor=2n-3$$ for every integer $n$ such that
    \begin{itemize}
        \item $q\nmid n$, and
        \item $2\leq n\leq k-1$,
    \end{itemize}
    and
    \item $$\sum_{i=1}^3\lfloor nx_i\rfloor+\lfloor\frac{nb}{q}\rfloor+\lfloor\frac{nc}{q}\rfloor=2n-4$$ for every integer $n$ such that
        \begin{itemize}
        \item $q\nmid n$,
        \item $n\in [k,\max\left\{2k-8,25\right\}]$, and 
        \item $n\not=k+1$.
    \end{itemize}
\end{itemize}
Then one of the following holds:
\begin{enumerate}
\item $\left(k,q,b,c\right)=\left(13,3,1,2\right)$, $x_1\in \left(\frac{1}{7},\frac{3}{20}\right)$, $x_2\in \left(\frac{7}{19},\frac{3}{8}\right)$, $x_3\in \left(\frac{2}{5},\frac{9}{22}\right)$. 
\item $\left(k,q,b,c\right)=\left(14,3,1,1\right)$, $x_1\in \left(\frac{2}{11},\frac{3}{16}\right)$, $x_2\in \left(\frac{2}{5},\frac{9}{22}\right)$, $x_3\in \left(\frac{15}{23},\frac{17}{25}\right)$. 
\item $\left(k,q,b,c\right)=\left(14,3,1,2\right)$, $x_1\in \left(\frac{2}{11},\frac{3}{16}\right)$, $x_2\in \left(\frac{8}{25},\frac{8}{23}\right)$, $x_3\in  \left(\frac{2}{5},\frac{9}{22}\right)$. 
\item $\left(k,q,b,c\right)=\left(16,3,1,2\right)$, $x_1\in \left(\frac{6}{25},\frac{1}{4}\right)$, $x_2\in \left(\frac{3}{10},\frac{7}{23}\right)$, $x_3\in \left(\frac{9}{23},\frac{2}{5}\right)$. 
\item $\left(k,q,b,c\right)=\left(16,3,1,2\right)$, $x_1\in \left(\frac{6}{25},\frac{1}{4}\right)$, $x_2\in \left(\frac{7}{23},\frac{4}{13}\right)$, $x_3\in \left(\frac{5}{13},\frac{9}{23}\right)$.  
\item $\left(k,q,b,c\right)=\left(16,3,1,2\right)$, $x_1\in \left(\frac{6}{25},\frac{1}{4}\right)$, $x_2\in \left(\frac{4}{13},\frac{5}{16}\right)$, $x_3\in \left(\frac{3}{8},\frac{5}{13}\right)$. 
\item $\left(k,q,b,c\right)=\left(17,3,1,2\right)$, $x_1\in \left(\frac{1}{10},\frac{2}{19}\right)$, $x_2\in \left(\frac{5}{13},\frac{9}{23}\right)$, $x_3\in \left(\frac{11}{25},\frac{9}{20}\right)$.
\item $\left(k,q,b,c\right)=\left(16,14,3,13\right)$, $x_1\in \left(\frac{2}{13},\frac{3}{19}\right)$, $x_2\in \left(\frac{7}{25},\frac{2}{7}\right)$, $x_3\in \left(\frac{8}{23},\frac{7}{20}\right)$. 
\item $\left(k,q,b,c\right)=\left(16,23,5,8\right)$, $x_1\in \left(\frac{2}{13},\frac{3}{19}\right)$, $x_2\in \left(\frac{7}{25},\frac{2}{7}\right)$, $x_3\in \left(\frac{13}{14},\frac{14}{15}\right)$. 
\item $\left(k,q,b,c\right)=\left(16,29,10,27\right)$, $x_1\in \left(\frac{2}{13},\frac{3}{19}\right)$, $x_2\in \left(\frac{5}{23},\frac{2}{9}\right)$, $x_3\in \left(\frac{7}{25},\frac{2}{7}\right)$.  
\end{enumerate}
\end{thm}

\begin{thm}\label{thm: 2 to 12 no solutions}
Find all real numbers $x_1,x_2,x_3$, such that
\begin{itemize}
    \item $0<x_1\leq x_2\leq x_3<1$, and
    \item $$\sum_{i=1}^3\lfloor nx_i\rfloor=n-2$$
    holds for every integer $n$ such that $n\in [2,12]$.
\end{itemize}
 Then we have no solution.
\end{thm}

\noindent\textbf{Explanation on algorithm of Theorem \ref{thm: 17equations no solution}}. For every integer $k\geq 2$, we let $V_k$ be the solution space $(x_1,x_2,x_3,x_4,x_5)\in [0,1)^5$ of 
$$\sum_{i=1}^5\lfloor nx_i\rfloor=2n-3,\forall\ 2\leq n\leq k.$$
Then $V_k$ is always a disjoint union of multiples of intervals. We find $V_k$ by induction on $k$. When $k=2$, the solution space is $$V_2=\left(\left[0,\frac{1}{2}\right),\left[0,\frac{1}{2}\right),\left[0,\frac{1}{2}\right),\left[0,\frac{1}{2}\right),\left[\frac{1}{2},1\right)\right).$$

Suppose that we have $V_k=\cup_{i=1}^mV_k^i$ for some integer $m\geq 0$, where $V_k^i$ are the irreducible components of $V_k$, and each $V_k^i$ is of the form 
$$\prod_{j=1}^5\left[a^i_{k,j,0},a^i_{k,j,1}\right).$$
Notice that for any $x\in \left[a^i_{k,j,0},a^i_{k,j,1}\right)$, and every integer $n\in [2,k]$, $\lfloor nx\rfloor$ is a constant. Thus 
\begin{itemize}
    \item either $\lfloor (k+1)x\rfloor$ is a constant for every $x\in \left[a^i_{k,j,0},a^i_{k,j,1}\right)$. In this case we let $V^i_{k,j}(0)=V^i_{k,j}(1):=\left[a^i_{k,j,0},a^i_{k,j,1}\right)$, or 
    \item there exists $a^i_{k,j,0}<c<a^i_{k,j,1}$ such that $(k+1)c\in\mathbb N^+$. In this case, let $V^i_{k,j}(0):=\left[a^i_{k,j,0},c\right)$ and $V^i_{k,j}(1):=\left[c,a^i_{k,j,1}\right)$.
\end{itemize}
Then there are at most $2^5=32$ possibilities (indeed it is much smaller practically) of $\prod_{j=1}^5V^i_{k,j}(\delta_j)$, where $\delta_j\in\left\{0,1\right\}$ for every $j$, such that $\sum_{j=1}^5\lfloor (k+1)x_j\rfloor$ is a constant along each set. We pick out those sets satisfying our requirements, denote $V_{k+1}$ to be the union of all of them, and continue the induction on $k$.

\begin{proof}[Proof of Theorem \ref{thm: A41349}]
For every $\frac{1}{r}(a_1,a_2,a_3,a_4,a_5)\in\cup_{r=1}^{51}\bar{\mathcal{A}}_r\left(4,\frac{1}{13}\right)$, since $4\times 13>51\geq r$, $\sum_{i=1}^5a_i=2r-1,2r-2$ or $2r-3$. Possibly reordering $a_1,a_2,a_3$ and $a_4,a_5$, we may assume that $0<a_1\leq a_2\leq a_3\leq r$ and $0<a_4\leq a_5\leq r$. Thus Theorem \ref{thm: 5dimcycrleq49} implies Theorem \ref{thm: A41349}. 
\end{proof}
\begin{proof}[Proof of Theorem \ref{thm: dn1218}]
Since $q:=q(x\in X)\geq 19$, $n\in\Ii_q$ for every integer $n\in [2,18]$. Thus Theorem \ref{thm: 17equations no solution} implies Theorem \ref{thm: dn1218}.
\end{proof}
\begin{proof}[Proof of Theorem \ref{thm: d314}]
Since $q:=q(x\in X)\geq 33$, we have $n\in\Ii_q$ for every integer $n\in [2,32]$. Thus Theorem \ref{thm: d314} follows from Theorem \ref{thm: k1318noq} by noticing that for every $\bm{x}=(x_1,x_2,x_3,x_4,x_5)$ as in the output of Theorem \ref{thm: k1318noq}, $\mathcal{D}(31,4)$ holds.
\end{proof}
\begin{proof}[Proof of Theorem \ref{thm: dn1332}]
Possibly reordering the coordinates of the associated point of $(x\in X)$, the theorem follows from \ref{thm: dn1332}.
\end{proof}
\begin{proof}[Proof of Theorem \ref{thm: d314d283}] Let $\bm{x}=\big(x_1,x_2,x_3,\frac{b}{q},\frac{c}{q}\big)$ be the associated point of $(x\in X)$. By Theorem \ref{thm: 17equations no solution divi6}, possibly reodering $x_1,x_2,x_3$, $(k,q,b,c,x_1,x_2,x_3)$ satisfies one the cases (1)--(10) of Theorem \ref{thm: 17equations no solution divi6}. The proof follows from the following:
\begin{itemize}
    \item In Case (1)(2), or in Case(3) and $x_2\in \big(\frac{9}{28},\frac{9}{26}\big)$, $\mathcal{C}(27)$ holds.
    \item In Case (3) and $x_2\in \left(\frac{8}{25},\frac{9}{28}\right)$ or in Case (8)(9)(10), $\mathcal{D}(31,4)$ holds.
    \item In Case (3) and $x_2\in \left(\frac{9}{26},\frac{8}{23}\right)$, $\mathcal{D}(31,3)$ and  $\mathcal{D}(35,2)$ holds.
    \item In Case (4)(5), $\mathcal{C}(33)$ holds.
    \item In Case (6), $\mathcal{C}(45)$ holds
    \item In Case (7), $k=16, q=3$, and $\mathcal{D}(31,3)$ holds.
\end{itemize}
\end{proof}

\begin{proof}[Proof of Theorem \ref{thm: 3dimcyc13}]
We have $q:=q(x\in X)=r\geq 13$. Thus $n\in\Ii_q$ for every integer $n\in [2,12]$. Since $a_4=1$ and $a_5=r-1$, for any integer $n\in [1,r-1]$, $$\lfloor\frac{na_4}{r}\rfloor+\lfloor\frac{na_5}{r}\rfloor=n-1.$$ Therefore, $\mathcal{D}(n,1)$ is equivalent to $$\sum_{i=1}^3\lfloor nx_i\rfloor=n-2$$ for every $n\in [2,12]$. Thus Theorem \ref{thm: 2 to 12 no solutions} implies Theorem \ref{thm: 3dimcyc13}.
\end{proof}

\begin{proof}[Proof of Theorem \ref{thm: d213}]
Suppose that $(x\in X)=\frac{1}{r}(a_1,a_2,a_3,a_4,a_5)$. Possibly reordering $a_1,a_2,a_3$, we may assume that $a_1\leq a_2\leq a_3$. Let $\bm{x}=(x_1,x_2,x_3,x_4,x_5)$ be the associated point of $(x\in X)$, then $x_4=x_5=\frac{1}{2}$. Since $\mld(x,X)>2-\frac{1}{4}$, $r>4$, hence $\gcd(a_4,r)>2$. Since $(x\in X)\in\bar{\mathcal{A}}(4)$,
\begin{equation}\label{equ: star}
x_i+x_j\in\left\{\frac{1}{2},\frac{3}{2}\right\} \tag{$\star$}
\end{equation}
for some $1\leq i\not=j\leq 3$. Let $g:\Ii_2\rightarrow\mathbb N$ be the function defined by $$g(n):=\sum_{i=1}^3\lfloor nx_i\rfloor$$ then $\mathcal{D}(n,1)$ holds if and only if $n\in\Ii_2$ and $g(n)=n-2$. 

Suppose that $\mathcal{D}(n,1)$ holds for $n\in\left\{3,5,7,9\right\}$. Since $g\left(3\right)=1$, $x_1\in \left(0,\frac{1}{3}\right), x_2\in \left(0,\frac{1}{3}\right)$, and $x_3\in \left[\frac{1}{3},\frac{2}{3}\right)$.

Since $g\left(5\right)=3$, by (\ref{equ: star}), one of the following holds:
\begin{itemize}
\item $x_1\in \left(0,\frac{1}{5}\right), x_2\in \left[\frac{1}{5},\frac{1}{3}\right)$, and $x_3\in \left[\frac{2}{5},\frac{3}{5}\right)$.
\item $x_1\in \left[\frac{1}{5},\frac{1}{3}\right), x_2\in \left[\frac{1}{5},\frac{1}{3}\right)$, and $x_3\in \left[\frac{1}{3},\frac{2}{5}\right)$.
\end{itemize}

Since $g\left(7\right)=5$, by (\ref{equ: star}), one of the following holds:
\begin{itemize}
\item  $x_1\in \left(0,\frac{1}{7}\right), x_2\in \left[\frac{2}{7},\frac{1}{3}\right)$, and $x_3\in \left[\frac{3}{7},\frac{4}{7}\right)$. 
\item $x_1\in \left[\frac{1}{7},\frac{1}{5}\right), x_2\in \left[\frac{2}{7},\frac{1}{3}\right)$, and $x_3\in \left[\frac{2}{5},\frac{3}{7}\right)$. 
\item $x_1\in \left[\frac{1}{5},\frac{2}{7}\right), x_2\in \left[\frac{2}{7},\frac{1}{3}\right)$, and $x_3\in \left[\frac{1}{3},\frac{2}{5}\right)$.
\end{itemize}

Since $g\left(9\right)=7$, one of the following holds:
\begin{itemize}
\item  $x_1\in \left(0,\frac{1}{9}\right), x_2\in \left[\frac{2}{7},\frac{1}{3}\right)$, and $x_3\in \left[\frac{5}{9},\frac{4}{7}\right)$.  
\item  $x_1\in \left(\frac{1}{9},\frac{1}{7}\right), x_2\in \left[\frac{2}{7},\frac{1}{3}\right)$, and $x_3\in \left[\frac{4}{9},\frac{5}{9}\right)$. 
\item $x_1\in \left[\frac{2}{9},\frac{2}{7}\right), x_2\in \left[\frac{2}{7},\frac{1}{3}\right)$, and $x_3\in \left[\frac{1}{3},\frac{2}{5}\right)$.  
\end{itemize}
But they all contradict to (\ref{equ: star}).
\end{proof}


\begin{thebibliography}{99}
	\bibitem{Ale93} V. Alexeev, \textit{Two two--dimensional terminations}, Duke Math. J., 69(3), 1993: 527--545. Res. Lett. \textbf{6}, no. 5--6, 573--580, 1999.
	
	\bibitem{Amb06} F. Ambro, \textit{The set of toric minimal log discrepancies}, Cent. Eur. J. Math. 4, no. 3, 358--370, 2006.
	
	\bibitem{Bor97} A.A. Borisov, \textit{Minimal discrepancies of toric singularities}, Manuscripta Mathematica 92(1), 1997: 33--45.
	
	\bibitem{CH20} G. Chen and J. Han, \textit{Boundedness of $(\epsilon,n)$-Complements for Surfaces}, arXiv:2002.02246v2, 2020.

\bibitem{HL20} J. Han and Y. Luo, \textit{On boundedness of divisors computing minimal log discrepancies for surfaces}, arXiv: 2005.09626v2, 2020.


\bibitem{HLS19} J. Han, J. Liu, and V.V. Shokurov, \textit{ACC for minimal log discrepancies for exceptional singularities}, arXiv: 1903.04338v2, 2019.

	
	\bibitem{Jia19} C. Jiang, \textit{A gap theorem for minimal log discrepancies of non-canonical singularities in dimension three}, arXiv: 1904.09642v1, 2019.

	\bibitem{Kaw11} M. Kawakita, {Towards boundedness of minimal log discrepancies by the Riemann--Roch theorem}. Amer. J. Math., 133(5), 2011: 1299--1311.
	
	
	\bibitem{Kaw14} M. Kawakita, \textit{Discreteness of log discrepancies over log canonical triples on a fixed pair}, Journal of Algebraic Geometry, 23(4), 2014: 765--774.

\bibitem{Kaw15} M. Kawakita, \textit{A connectedness theorem over the spectrum of a formal power series ring}, Int. J.
Math. \textbf{26}, No. 11, Article ID 1550088, 27p. (2015)

	\bibitem{Kaw18} M. Kawakita, {On equivalent conjectures for minimal log discrepancies on smooth threefolds}. To appear in J. Algebraic Geom., arXiv: 1803.02539, 2018.


	\bibitem{Liu18} J. Liu, \textit{Toward the equivalence of the ACC for a-log canonical thresholds and the ACC for minimal log discrepancies}. arXiv: 1809.04839v3, 2018.

\bibitem{LX19} J. Liu and L. Xiao, \textit{An optimal gap of minimal log discrepancies of threefold non-canonical singularities} (The first version of this paper with precise algorithms), arXiv: 1909.08759v1, 2019.

\bibitem{MN18} Mircea Musta\c{t}\v{a} and Yusuke Nakamura, {A boundedness conjecture for minimal log discrepancies on a fixed germ} in Local and global methods in algebraic geometry, Contemp. Math. (712), 287--306, 2018.



\bibitem{Nak16} Y. Nakamura, \textit{On minimal log discrepancies on varieties with fixed Gorenstein index}, Michigan Math. J., Volume 65, Issue 1, 165--187, 2016.

		\bibitem{Sho94} V.V. Shokurov, \textit{A.c.c. in codimension 2}, 1994 (preprint).
	
		\bibitem{Sho96} V.V. Shokurov, {3--fold log models}. J. Math. Sciences \textbf{81}, 2677--2699, 1996.
		
		\bibitem{Sho00} V.V. Shokurov, \textit{Complements on surfaces}. J. Math. Sci. (New York), 102, no. 2 (2000): 3876--3932.

	\bibitem{Sho04} V.V. Shokurov, \textit{Letters of a bi-rationalist}, V. Minimal log discrepancies and termination of log flips. (Russian) Tr. Mat. Inst. Steklova 246,  Algebr. Geom. Metody, Svyazi i Prilozh., 328--351, 2004.			
	
\end{thebibliography}
\end{document}